\documentclass[12pt]{article}

\usepackage{amsmath,amsfonts,amssymb,amsthm,amsrefs}
\usepackage{enumitem}

\title{Reverse mathematics and colorings of hypergraphs}
\author{C. Davis, J. Hirst\footnote{Corresponding author:  Jeffry Hirst, Department of Mathematical Sciences, Appalachian State University, Boone, NC 28608  hirstjl@math.appstate.edu  ORCID:0000-0002-8273-8951\newline
This is a pre-print of an article published in the Archive for Mathematical Logic. The final authenticated version is available online at: https://doi.org/10.1007/s00153-018-0654
},
J. Pardo, and T. Ransom
}

\date{April 5, 2018}

\setlist[enumerate]{label=\rm{(\arabic*)}, ref=\arabic*}
\theoremstyle{plain}
\newtheorem{thm}{Theorem}
\newtheorem{lemma}[thm]{Lemma}

\theoremstyle{definition}
\newtheorem*{defn}{Definition}      

\theoremstyle{definition}

\newcommand{\nat}{\mathbb N}  
\newcommand{\hypg}[2]{\langle #1,#2 \rangle}
\newcommand{\rca}{{\sf RCA}_0}
\newcommand{\wkl}{{\sf WKL}_0}
\newcommand{\aca}{{\sf ACA}_0}
\newcommand{\poo}{{\Pi^1_1 \text{-}}{\sf CA}_0}
\newcommand{\lh}{{\text {length}}}
\newcommand{\srt}{{\sf {SRT}}^2_2}
\newcommand{\seq}[1]{\langle #1_i \rangle_{i \in \nat}}
\newcommand{\cat}{{}^\smallfrown}
\newcommand{\mred}{{\text{red}}}
\newcommand{\mblue}{{\text{blue}}}

\begin{document}

\maketitle

\begin{abstract}
Working in subsystems of second order arithmetic, we formulate several representations
for hypergraphs.  We then prove the equivalence of various vertex coloring theorems to
$\wkl$, $\aca$, and $\poo$.  MSC: 03B30; 03F35
\end{abstract}

A hypergraph consists of a set of vertices together with a set of edges.  While edges in graphs always
have exactly two vertices, hypergraphs may have edges of any cardinality, finite or infinite.  Some authors
exclude edges that are empty of have a single vertex, but for this article, edges may be of any size.
After considering alternatives for the representations of edges, we will turn to results on the strength of
theorems related to three types of vertex colorings.

\subsection*{Representations of edges}

The edges of hypergraphs can be presented in a variety of ways.  If all the edges are finite,
each edge can be encoded by a single number.  In this case, the edges can be represented
by a set of codes or by a sequence of codes.  Finite edges and infinite edges can both be
represented by a sequence of characteristic functions.  Some changes in representation can be
carried out in $\rca$, while others require additional set comprehension, as shown by the next
four results.

\begin{thm}\label{rep1}
$(\rca )$ If $H$ is a hypergraph with finite edges represented by a set of edges, then there is a hypergraph $H'$ with exactly the same edges represented by a sequence of its edges, possibly with repetitions.	
\end{thm}

\begin{proof}
We argue in $\rca$. Suppose $H=\left<V,E\right>$ is a hypergraph with finite edges, where $E$ is a set of integer codes for the edges of $H$. Let $e_0$ be an integer code for an edge of $H$. Define the function $f:\mathbb{N}\rightarrow\mathbb{N}$ by
$$f(n) =  \begin{cases} n & n\in E \\ e_0 & \text{otherwise} \end{cases}$$
Then the range of $f$ is $E$, and $V$ together with $f$ represents $H$ using a sequence of codes.
\end{proof}

\begin{thm}\label{rep2}
$(\rca)$ If $H$ is a hypergraph with finite edges represented by a sequence of its edges, then there is a hypergraph $H'$ with exactly the same edges represented by a sequence of characteristic functions for its edges.
\end{thm}

\begin{proof}
We argue in $\rca$. Suppose $H=\left<V,E\right>$ is a hypergraph with finite edges, where $E$ is the sequences for the edges of $H$. Define the characteristic function $\chi_i$ by
$$\chi_i(v_j) =  \begin{cases} 1 & v_j\in \left<e_i\right> \\ 0 & \text{otherwise} \end{cases}$$
By recursive comprehension, the sequence of characteristic function $\left<\chi_i\right>_{i\in\mathbb{N}}$ exists. The vertices $V$ together with the sequence of characteristic functions for each edge represent $H$.
\end{proof}

\begin{thm}\label{rep3}
$(\rca )$  The following are equivalent:
\begin{enumerate}
\item $\aca$. \label{rep31}
\item If $H$ is a hypergraph with finite edges represented by a sequence of edges, then $H$ can be represented by a set of edges.\label{rep32}
\end{enumerate}
\end{thm}

\begin{proof}
First we will prove that (\ref{rep31}) implies (\ref{rep32}). Reasoning in $\aca$, let $H$ be a hypergraph with finite edges represented by the sequence $\left<e_i\right>_{i\in\mathbb{N}}$. Arithmetical comprehension can prove the edge set $\{e\mid\exists i(e=e_i)\}$ exists.
	
To prove the converse, by Lemma III.1.3 of Simpson \cite{simpson}, it suffices to use item (\ref{rep32}) to prove the existence of the range of an injection. Let $g:\mathbb{N}\rightarrow\mathbb{N}$ be an injection. Define the hypergraph $H=\left<V,E\right>$ as follows. Let $\mathbb{N}$ be the set of vertices. Define a sequence of edges $\left<e_n\right>_{n\in\mathbb{N}}$ by $e_n=\{0,g(n)+1\}$.
Note that $m$ is in the range of $g$ if and only if the set $\{0,m+1\}$ is an edge of $H$. Given the set of edges of $H$, recursive comprehension proves the existence of the range of $g$.	
\end{proof}

\begin{thm}\label{rep4}
$(\rca )$ The following are equivalent:
\begin{enumerate}
\item $\aca$. \label{rep41}
\item If $H$ is a hypergraph with finite edges represented by a sequence of characteristic functions, then $H$ can be represented by a sequence of finite set codes for edges.\label{rep42}
\end{enumerate}
\end{thm}

\begin{proof}
We begin by proving that \ref{rep41} implies \ref{rep42}. Reasoning in $\aca$, let $H$ be a hypergraph with finite edges represented by the sequence of characteristic functions $\left<e_i\right>_{i\in\mathbb{N}}$. Define $s_i=\{j\mid e_i(j)=1\}$ for $i\in\mathbb{N}$. Then the sequence $\left<s_i\right>_{i\in\mathbb{N}}$ is arithmetically definable.
	
To prove the converse, it suffices to use statement \ref{rep42} to prove the existence of a range of an injection. Let $g:\mathbb{N}\rightarrow\mathbb{N}$ be that injection. Let $H$ be the hypergraph with vertex set $\mathbb{N}$ and edges defined by the sequence of characteristic functions $\left<e_i\right>_{i\in\mathbb{N}}$ defined as follows	
$$e_i(n) = \begin{cases} 1 & n=2i \vee n=2i+2 \\ 0 & n=2j\wedge j\not\in \{i,i+1\} \\ 1 & n=2j+1\wedge g(j)=i \\ 0 & n=2j+1\wedge g(j)\neq i \end{cases}$$
The recursive comprehension axiom proves the existence of $\left<e_i\right>$ using $g$ as a parameter. Applying the principle of statement 2, let $\left<s_i\right>_{i\in\mathbb{N}}$ be a sequence of finite set codes for the edges of $H$. Given any value $y$, successively examine the set codes until we locate the unique $s_i$ encoding a set containing $2y$ and $2y+2$. If this set contains only $2y$ and $2y+2$, then $y$ is not in the range of $g$, but if the set also contains another element, then $y$ is in the range of $g$. Thus, recursive comprehension proves the existence of the range of $g$. 
\end{proof}

In order to state compactness results for vertex colorings, we need to consider finite substructures
of hypergraphs.
The following terminology is based on that of Berge \cite{berge}.  When constructing substructures
of a hypergraph $H$, one can either require that all edges in the substructure are edges of $H$, or allow
edges of the substructure to be subsets of edges of $H$.  

\begin{defn}
A {\sl partial hypergraph} of a hypergraph $H$ consists of a subset $V$ of the vertices
of $H$ and a subset $E$ of the edges of $H$ such that for every edge $e \in E$, $e \subset V$.
\end{defn}

\begin{defn}
A {\sl partial subhypergraph} of a hypergraph $H$ consists of a subset $V$ of the vertices of $H$ and a set $E$ of
subsets of $V$ such that for every $e\in E$, there is an edge $e^\prime \in H$ such that $e = e^\prime \cap V$.
\end{defn}

\subsection*{Vertex colorings and finite edges}

A vertex coloring is a map from the vertices of a hypergraph into a set of colors, usually coded
by a subset of $\nat$.
In the literature, most definitions for vertex colorings are stated for finite hypergraphs,
but those listed below extend naturally to
infinite hypergraphs.  The definition for
conflict-free colorings is based on that of Smorodinsky \cite{smo}.  The other definitions can be found in the book of Berge \cite{berge}.

\begin{defn}  Suppose $\hypg{V}{E}$ is a hypergraph.  Let $\chi : V \to \nat$ be a coloring of the vertices.  We say:
\begin{enumerate}
\item  $\chi$ is a $k$-coloring if the range of $\chi$ is contained in $\{0, 1, \dots k-1\}$.
\item  $\chi$ is {\sl proper} if $\chi$ is non-constant on each edge that contains more than one vertex.
\item  $\chi$ is {\sl strong} if $\chi$ is injective on each edge.
\item $\chi$ is {\sl conflict-free} \cite{smo} if each edge contains one vertex whose color matches no other vertex in the edge.
That is, $\forall e \exists j ~ | \{ v \in e \mid \chi (v) = j \}| = 1$.
\end{enumerate}
\end{defn}

Theorem 3.4 of Hirst \cite{hirstmt} states that $\wkl$ is equivalent to the statement that every
locally $k$-colorable graph has a $k$-coloring.  The following result generalizations this theorem
to the hypergraph setting, for hypergraphs with sequences or sets of finite edges.

\begin{thm}\label{pcffA}
$(\rca)$  For $k \ge 2$, the following are equivalent:
\begin{enumerate}
\item $\wkl$.\label{pcffA1}
\item Let $H$ be a hypergraph with a sequence of finite edges.  If every finite
partial hypergraph of $H$ has a proper $k$-coloring, then $H$ has a proper $k$-coloring.\label{pcffA2}
\item Statement {\rm(\ref{pcffA2})} with ``proper'' replaced by ``conflict-free''.\label{pcffA3}
\item Statement {\rm(\ref{pcffA2})} with ``sequence of edges'' replaced by ``set of edges.''\label{pcffA4}
\item Statement {\rm(\ref{pcffA3})} with ``sequence of edges'' replaced by ``set of edges.''\label{pcffA5}
\end{enumerate}
\end{thm}

\begin{proof}
To prove that (\ref{pcffA1}) implies (\ref{pcffA2}), assume $\wkl$ and let $H$ be a hypergraph with vertex set $\{v_1, v_2, \ldots\}$ such that every finite partial hypergraph of $H$ has a proper $k$-coloring.  For every $n$, let $m_n$ be the least integer such that $m_n > m_{n-1}$ and if a vertex $v_i$ appears in one of the first $n$ edges of $H$ then $i \le m_n$.  Let $H_n$ be the finite partial hypergraph with vertex set $\{v_1, v_2, \ldots, v_{m_n}\}$ and the first $n$ edges of $H$.  Let $T$ be the tree consisting of those finite sequences $\sigma$ in $k^{<\nat}$ such that whenever length$(\sigma)=m_n$, $\sigma$ is a proper $k$-coloring of $H_n$.  Since every finite partial hypergraph of $H$ has a proper $k$-coloring, $T$ must contain infinitely many sequences.  By $\wkl$, $T$ has an infinite path.  This path yields a proper $k$-coloring of $H$.

To prove that (\ref{pcffA1}) implies (\ref{pcffA3}), repeat the previous proof replacing every instance of the word ``proper'' by ``conflict-free.''  Theorem \ref{rep1} shows that (\ref{pcffA2}) implies (\ref{pcffA4}) and (\ref{pcffA3}) implies (\ref{pcffA5})
It remains to show that both (\ref{pcffA4}) and (\ref{pcffA5}) imply (\ref{pcffA1}).  For graphs, that is, for hypergraphs in which every edge
has exactly two vertices, every coloring is proper if and only if it is conflict-free.  Thus, both (\ref{pcffA4}) and (\ref{pcffA5}) imply that if every finite
subgraph of a graph $H$ is $k$-colorable, then $H$ can be $k$-colored.  This implies $\wkl$ by Theorem 3.4 of Hirst \cite{hirstmt}.
\end{proof}

For the analog of the previous theorem for 
graphs with edges represented by sequences of characteristic functions, additional set comprehension strength is required.

\begin{thm}\label{conj2}
$(\rca )$  For any $k\ge 2$, the following are equivalent.
\begin{enumerate}
\item \label{conj21} $\aca$.
\item \label{conj22} Suppose $H$ is a hypergraph with finite edges given by a sequence of characteristic functions.  If every finite
partial hypergraph of $H$ has a proper $k$-coloring then $H$ has a proper $k$-coloring.
\item \label{conj23} Statement {\rm(\ref{conj22})} with ``proper'' replaced by ``conflict-free.''
\end{enumerate}
\end{thm}

\begin{proof}
To prove that (\ref{conj21}) implies (\ref{conj22}), fix $k$, assume $\aca$, and suppose $H$ is a hypergraph with finite edges
given by a sequence of characteristic functions.  Further suppose that every finite partial hypergraph of $H$ can be properly $k$-colored.
By Theorem \ref{rep4}, we can find the sequence of edges for $H$.  By Theorem (\ref{pcffA}) and the fact that $\aca$ implies $\wkl$,
$H$ has a proper $k$-coloring.

The proof that (\ref{conj21}) implies (\ref{conj23}) is similar to the preceding argument, with proper replaced by conflict-free.

For the reversals, we begin by considering the proof that (\ref{conj22}) implies (\ref{conj21}) for $k=2$.  By Lemma III.1.3 of Simpson
\cite{simpson}, it suffices to use $\rca$ and (\ref{conj22}) to prove the existence of the range of an arbitrary injection $f$.  The vertices
of our hypergraph $H$ will be $\{ b_ i \mid i \in \nat \} \cup \{ v_{i,j} \mid i,j \in \nat \}$.  The edges of $H$ will be represented by
characteristic functions for the following sets.  For all $n$ and $i$, $p_{n,i} = \{ v_{n,i} , v_{n,i+1} \}$,
$q_n = \{ b_n , b_{n+1} \}$,
$r_n = \{ b_{2n} , v_{f(n), 2n} \}$, and
$s_n =\{ v_{n,0}, b_{2n+1} \} \cup \{ b_{2i} \mid f(i) = n\}$.
Because $f$ is an injection, $b_{2i}$ is in $s_n$ if and only if $\exists t< 2i (f(t) = n )$, so the characteristic function
for $s_n$ is uniformly computable from $f$.  Via dovetailing, recursive comprehension suffices to prove the existence
of a sequence of characteristic functions for all these edges of $H$.
Note that edges of the form $s_n$ contain two vertices if $n$ is not in the range of $f$ and three vertices if $n$ is in the range.
All other edges contain exactly two vertices.

We claim that every finite partial subhypergraph of $H$ has a proper 2-coloring.  Let $G$ be a partial subhypergraph.
Let $n_0$ be the maximum natural number such that $b_{2n_0}$ is in $G$. Define the coloring $\chi$ by
$\chi (b_n ) = n \mod 2$ if $n\le 2n_0 + 1$, $\chi (v_{n,j} )= j \mod 2$ if $\forall t \le n_0 (f(t) \neq n)$,
and $\chi (v_{n,j} ) = j+1 \mod 2$ if $\exists t \le n_0 ~(f(t) = n)$.  The first clause guarantees that $\chi$ is a proper
coloring with regard to every edge of the form $q_n$, and the last two clauses insure that $\chi$ is proper on edges of
the forms $p_{n,i}$, $r_n$ and $s_n$ restricted to the vertices of $G$.

Apply (\ref{conj22}) to obtain a proper 2-coloring of $H$.  If necessary, permute the colors so that $\chi( b_0 ) = 0$.
Fix $n$.  If $f(t) = n$ for some $t$, then $r_t = \{ b_{2t} , v_{n,2t} \}$ is in $H$, so $\chi (v_{n,0} ) = 1$.  If $n$ is not
in the range of $f$, then $s_n = \{v_{n,0}, b_{2n+1}\}$ is in $H$, so $\chi (v_{n,0} ) = 0$.  Thus
the range of $f$ is defined by $\{ n \mid \chi (v_{n,0} ) = 1\}$, which exists by recursive comprehension.

The proof that (\ref{conj22}) implies (\ref{conj21}) can be extended to values of $k$ greater than $2$ by adding
a complete graph on $k-2$ vertices to $H$ and connecting each vertex in the complete graph to every vertex in
the prior construction.  Finally, for graphs with edges of size at most 3, any coloring is conflict-free if and only if
it is proper.  Thus (\ref{conj23}) implies (\ref{conj21}) can be proved by replacing all uses of ``proper'' in the
preceding arguments with ``conflict-free.''
\end{proof}

For strong colorings, edge representation does not affect the strength of the coloring statements.
Additionally, in some settings, finite colorability suffices to imply $\wkl$ over $\rca$.  This provides an interesting
contrast to the case for graphs, as described following the proof of the next result.

\begin{thm}\label{sfcA}
$(\rca)$  The following are equivalent:
\begin{enumerate}
\item $\wkl$.\label{sfcA1}
\item Let $H$ be a hypergraph with any edge representation.  If there is a $k$ such that every finite
partial hypergraph of $H$ has a strong $k$-coloring, then $H$ has a strong $k$-coloring.\label{sfcA2}
\item Let $H$ be a hypergraph with a set of finite sets for edges.  If every finite partial hypergraph
of $H$ has a strong $3$-coloring, then $H$ has a strong $k$-coloring for some $k$.\label{sfcA3}
\item Let $H$ be a hypergraph with a sequence of finite sets for edges.  If every finite partial hypergraph
of $H$ has a strong $2$-coloring, then $H$ has a strong $k$-coloring for some $k$.\label{sfcA4}
\end{enumerate}
\end{thm}

\begin{proof}
To prove that (\ref{sfcA1}) implies (\ref{sfcA2}), assume $\wkl$
and let $H$ be a hypergraph.  By Theorem \ref{rep1} and Theorem \ref{rep2},
we may assume that the edges of $H$ are given by a sequence of characteristic functions.
Fix $k$ and suppose that for every $n$, the partial subhypergraph given by the first $n$ values of each of the
the first $n$ characteristic functions has a strong $k$-coloring.  Let $T$ be the tree consisting of those finite
sequences $\sigma$ in $k^{<\nat}$ such that $\sigma$ is a strong $k$-coloring of the partial subhypergraph
defined by the first $\lh (\sigma )$ many values of the first $\lh (\sigma )$ many edge characteristic functions.
Because the finite partial hypergraphs of $H$ are colorable, $T$ must contain infinitely many sequences.
By $\wkl$, $T$ has an infinite path.  This path yields a strong $k$-coloring of $H$.

Note that (\ref{sfcA3}) follows immediately from from (\ref{sfcA2}) restricted to $k=3$.  To prove that (\ref{sfcA3}) implies
(\ref{sfcA1}), by Lemma IV.4.4 of Simpson \cite{simpson}, it suffices to use (\ref{sfcA3}) to separate the ranges
of injections with disjoint ranges.  Fix injections $f$ and $g$ such that $\forall m \forall n (f(m) \neq g(n))$ and
construct a hypergraph $H$ as follows.
The vertices of $H$ consist of the sets $U=\{u _i \mid i \in \nat \}$ and $V=\{v_i \mid i \in \nat\}$.
For each triple $i$, $j$, and $k$, the edge $\{ u_i , u_j , v_k \}$ is in $H$ if and only if $i<k$, $j<k$,
$\exists t<k ~f(t)=i$, and $\exists t<k ~g(t) = j$.  Note that each edge contains exactly three vertices and
exactly one vertex from $V$.  No edge contains two vertices from $U$ indexed by elements of the range of $f$, and
similarly, pairs of vertices from $U$ indexed by elements of the range of $g$ are forbidden within an edge.
Let $H_0$ be a finite partial hypergraph consisting of vertices $\{u _{i_0}, \dots u_{i_m} \}$ and
$\{v_{j_0} , \dots v_{j_n} \}$, and some or all of the edges of $H$ with vertices entirely contained in these sets.
The coloring that assigns color $0$ to all vertices $u_{i_k}$ such that
$\exists t<j_n ~ f(t) = i_k $, color $1$ to all vertices $u_{i_k}$ such that
$\exists t < j_n ~ g(t) = i_k$, and color $3$ to all remaining vertices is a strong $3$-coloring of $H_0$.
Thus, every finite partial hypergraph of $H$ has a strong $3$-coloring.  Apply (\ref{sfcA3}) to obtain a
$k$-coloring of $H$.  Note that if $\exists t ~f(t) = i$ and $\exists t ~g(t) = j$, then $u_{i}$ and $u_{j}$
are included in an edge, and so $u_{i}$ and $u_{j}$ must have distinct colors.  By bounded $\Sigma^0_1$
comprehension (which is provable in $\rca$ by Remark II.3.11 of Simpson \cite{simpson}) the set
$C = \{ i < k \mid \exists t ~\chi (u_{f(t)} ) = i \}$ exists.  $\rca$ proves that the set $\{ j \in \nat \mid \chi (u_j ) \in C \}$ exists,
contains the range of $f$, and is disjoint from the range of $g$.

Item (\ref{sfcA4}) also follows immediately from (\ref{sfcA2}).  As in the previous paragraph, to prove that
(\ref{sfcA4}) implies $\wkl$, let $f$ and $g$ be injections with disjoint ranges.  Construct $H$ as follows.
The vertices of $H$ are $V = \{ v_i \mid i \in \nat \}$.  Using a bijective pair encoding, for each $i \in \nat$,
let $i_0$ and $i_1$ denote the components of the pair encoded by $i$.  For each $i$, define the edge
$e_i = \{ v_{f(i_0 )} , v_{g(i_1 )} \}$.  Let $E_0$ be any finite set of edges with indices bounded by $b$, and
let $V_0$ be some set of vertices containing all the vertices in edges in $E_0$.  The coloring which assigns
color 0 to all vertices of the form $v_{f(t)}$ for $t<b$ and color 1 to all other vertices of $V_0$ is a strong
$2$-coloring of the finite partial hypergraph defined by $E_0$ and $V_0$.  By (\ref{sfcA4}), $H$ has a strong
$k$-coloring for some $k$;  call it $\chi$.  By bounded $\Sigma_1^0$ comprehension, provable in $\rca$ by Exercise II.3.13 of
Simpson \cite{simpson}, the set $K = \{ j<k \mid \exists t ~ \chi( v_{f(t) }) = j\}$ exists.  By recursive comprehension,
the set $S = \{ n \mid \chi (v_n ) \in K \}$ exists.  By the construction of $H$, $S$ contains all elements of the range
of $f$ and excludes all elements of the range of $g$.
\end{proof}

In the preceding theorem, if the number 3 is replaced by 2 in item (\ref{sfcA3}), the resulting statement asserts
that every locally $2$-colorable graph has a finite coloring.  This statement is known to be equivalent to $\wkl$
over $\rca$ plus $\Sigma^0_2$ induction, but the equivalence over $\rca$ is one of many open questions
listed in section \S5 of Dorais, Hirst, and Shafer \cite{dhs}.  Item (\ref{sfcA4}) shows
that for graphs with edges presented as a sequence rather than as a set, the statement that every locally $2$-colorable
graph has a finite coloring is equivalent to $\wkl$ over $\rca$.

\subsection*{Vertex colorings and infinite edges}

Allowing hypergraphs with infinite edges significantly affects the nature of vertex colorings.  Clearly, any
strong vertex coloring of a hypergraph with an infinite edge must use infinitely many colors.  Hypergraphs
with infinite edges may or may not have finite proper or conflict-free colorings.  We will show that sorting
those graphs with finite colorings from those without requires $\poo$.  The next definition and theorem assist
in that proof.

\begin{defn}
If $T$ is a tree, the set $L$ of {\sl leaves of }$T$ consists of those sequences in $T$ which have no extensions.
That is, $L = \{ \sigma \in T \mid \forall n ~\sigma \cat n \notin T\}$.
\end{defn}

\begin{lemma}\label{leaves}
$(\rca )$  The follow are equivalent.
\begin{enumerate}
\item $\poo$.\label{leaves1}
\item  If $\seq T$ is a sequence of trees in $\nat^{<\nat}$, then there is a function $f: \nat \to 2$ such that
$f(i) = 1$ if and only if $T_i$ contains an infinite path.\label{leaves2}
\item  If $\seq T$ is a sequence of trees and $\seq L$ is a sequence of sets such that for each $i$, $L_i$ is the set
of leaves of $T_i$, then there is a function $f: \nat \to 2$ such that
$f(i) = 1$ if and only if $T_i$ contains an infinite path.\label{leaves3}
\end{enumerate}
\end{lemma}

\begin{proof}
We work in $\rca$.  The equivalence of (\ref{leaves1}) and (\ref{leaves2}) is Lemma VI.1.1 of Simpson \cite{simpson}.
Because (\ref{leaves3}) is a special case of (\ref{leaves2}), we need only prove that (\ref{leaves3}) implies (\ref{leaves2}).

Let $\seq T$ be a sequence of trees as in (\ref{leaves2}).  For any sequence $\sigma$ of positive integers, let
$\sigma_{-1}$ denote the sequence of the same length as $\sigma$ such that for all $i$, $\sigma_{-1}(i) = \sigma (i)-1$.
For each $i$, define the tree ${\hat T}_i$ by letting $\tau \in {\hat T}_i$ if and only if either $\tau_{-1} \in T_i$ or
$\tau = \sigma\cat 0$ and $\sigma_{-1} \in T_i$.  For each $i$, the leaf set of $\hat T_i$ consists precisely of those
sequences of the form $\sigma_{-1} \cat 0$ such that $\sigma \in T_i$.  $\rca$ can prove that $\seq{\hat T}$
and $\seq L$ exist.  Note that $p$ is an infinite path in $\hat T_i$ if and only if $p_{-1}$ is an infinite path in $T$.
Thus, the function $f$ satisfying (\ref{leaves3}) for $\seq{\hat T}$ also satisfies (\ref{leaves2}) for $\seq T$.
\end{proof}

The preceding result allows us to prove the reversal below in a single step, rather than using (\ref{p2}) to prove
$\aca$ and then deducing $\poo$ in a second step.

\begin{thm}
$(\rca )$  For each $k \ge 2$, the following are equivalent.
\begin{enumerate}
\item  $\poo$.\label{p1}
\item  If $\seq H$ is a sequence of hypergraphs, then there is a function $f: \nat \to 2$ such that
$f(i) = 1$ if and only if $H_i$ has a proper $k$-coloring.\label{p2}
\item  Statement $(\ref{p2})$ with ``proper'' replaced by ``conflict-free.''\label{p3}
\end{enumerate}
\end{thm}

\begin{proof}
Assume $\rca$.  To prove that (\ref{p1}) implies (\ref{p2}), fix $k \ge 2$ and suppose
$\seq H$ is a sequence of hypergraphs.  Assuming that the vertices of $H$ are a subset of $\nat$,
the statement ``$H$ has a proper $k$-coloring'' asserts the existence of a function $g:\nat \to k$ such that
$g$ is not constant on any edge.  Thus, the set of indices $i$ such that $H_i$ has a proper $k$-coloring
is definable by a $\Sigma^1_1$ formula.  Applying $\poo$, the complement of this set of indices
exists, so by recursive comprehension, the set of indices and its characteristic function also exist.  This
function satisfies item (\ref{p2}) of the theorem.  A similar argument shows that (\ref{p1}) implies (\ref{p3}).

Next, we will prove that (\ref{p2}) implies (\ref{p1}) for the case $k=2$, indicating parenthetically how to
modify the argument for (\ref{p3}) implies (\ref{p1}).  By Lemma \ref{leaves}, it suffices to use (\ref{p2})
to determine which trees are well-founded in a list of trees with leaf sets.  Let $\seq{T}$ be a sequence of trees
in $\nat ^{<\nat}$, and let $\seq{L}$ be the corresponding leaf sets.  Given any tree
$T \subset \nat^{\nat}$ with leaf set $L$, define a hypergraph $H$ as follows.
The vertices of $H$ are $\{ a_0 , a_1, b_0, b_1 , s\}$ together with vertices labeled
$\sigma_0$ and $\sigma_1$ for each nonempty sequence $\sigma$ in $T$.  (One can routinely
assign integer codes to these vertices, and arrange for the set of codes to be $\nat$ or an initial segment of $\nat$,
if so desired.)
The edges of $H$ consist of
\begin{list}{$\bullet$}{}
\item  $(a_0, a_1 )$, $(a_1, s)$, $(b_0 ,b_1 )$, and $(b_1 , s)$,
\item  $(\sigma_0 , \sigma_1 )$ for every nonempty $\sigma \in T$,
\item  $(\sigma_1 , s)$ if $\sigma$ is a leaf of $T$,
\item $E_\sigma = \{ \sigma_1 \}\cup\{ \tau_0  \mid \tau\in T \land \exists n ~ \tau = \sigma\cat n\}$ if $\sigma\in T$ is not a leaf, and
\item $E_0 = \{a_0, b_0\} \cup \{ \sigma_0 \mid \sigma\in T \land \lh (\sigma ) = 1\}$.
\end{list}
Using $T$ and $L$ as parameters, $\rca$ can prove the existence of $H$ uniformly.
Note that $E_0$ and $E_\sigma$ may be infinite edges.

Suppose $c: \nat \to \{ \mred,~ \mblue \}$ is any $2$-coloring of the vertices of $H$.
Swapping colors if necessary, let $s$ be blue.  Assume $c$ is proper.  Then $a_1$ and $b_1$ are red
and $a_0$ and $b_0$ are blue.  Also, for any $\sigma \in T$ if $\sigma_0$ is red, then $\sigma_1$ is blue,
so $\{ \sigma_1 ,s \}$ is not an edge and $\sigma$ is not a leaf.  Because $\sigma_1$ is blue, by the definition
of $E_\sigma$, for some immediate successor $\tau$ of $\sigma$, $\tau_0$ must be red.  Finally,
to properly color $E_0$, for some $\sigma \in T$ of length 1, $\sigma_0$ is red.  Let $n$ be the least value such
that for $\tau = \sigma\cat n$, $\tau_0$ is red.  Iterating this process traces an infinite path in $T$.  Summarizing,
if $c$ properly $2$-colors $H$, then $T$ has an infinite path.  (Every conflict-free coloring is proper, so if $c$ is
conflict-free then $T$ has an infinite path.)

Conversely, suppose that $T$ has an infinite path $p$, and let $\sigma^0 \subset \sigma^1 \subset \sigma^2 \subset \dots$ be
the nonempty initial segments of $p$.  Define $c$ by $c(s)=c(a_0) = c(b_0) = \mblue$,
$c(a_1)=c(b_1)=\mred$, $\sigma^i_0 = \mred$ and $\sigma^1_1 = \mblue$ for $\sigma^i$ in $p$, and
$\tau_0 = \mblue$ and $\tau_1 = \mred$ for $\tau\in T$ not in $p$.  Treating the definitions of the edges of $H$ as
cases, one can verify that $c$ is a proper (conflict-free) $2$ coloring of $H$.  Summarizing the last two paragraphs,
$H$ has a proper (conflict-free) $2$-coloring if and only if $T$ has an infinite path.

Carrying out the the construction uniformly for all the trees in $\seq{T}$ and applying (\ref{p2})
of the theorem, we can find the characteristic function for the well-founded trees, as desired.
The parenthetical comments show that for $k=2$, (\ref{p3}) implies (\ref{p1}).

For values of $k>2$, modify the construction of the previous reversal by adding a complete graph
with $k-2$ vertices to $H$, connecting each vertex of the complete graph to every vertex of $H$ by
an edge consisting of two vertices.  The resulting hypergraph has a proper (conflict-free) $k$-coloring if and only
if $T$ is not well-founded.
\end{proof}

It would be interesting to know if the preceding result continues to hold if ``$k$-coloring'' is replaced by
``finite coloring.''
Lemma \ref{leaves} may prove useful in recasting the preceding theorem as a result related
to $\Sigma^1_1$ completness using many-one reducibility, or as a result on Weihrauch reducibility.

Infinite edges can interfere with finite conflict-free colorings of hypergraphs.  The following graph illustrates this situation.

\begin{defn}\label{Mdefn}
The $\cal M$-graph (the Matryoshka graph) is the hypergraph with vertex set $\nat$ and edges $\{ E_j \mid j\in \nat\}$
where $E_j = \{ k \mid j\le k\}$.
\end{defn}

Every finite partial subhypergraph of the $\cal M$-graph has a conflict-free 2-coloring.  For example, given any finite
collection of vertices, simply color the largest numbered vertex red and all other vertices blue.  On the other hand,
the entire $\cal M$-graph has no finite conflict-free coloring.  This can be proved directly by induction on $\Sigma^0_2$
formulas, or by using the following lemma.

\begin{lemma}$(\rca )$\label{ert}
The following are equivalent.
\begin{enumerate}
\item  $\sf{ERT}$ (eventually repeating tails):  Suppose $f : \nat \to k$ for some $k \in \nat$.\label{ert1}
Then there is a $b \in \nat $ such that  for all $ x \ge b $ there is a $y \ge b$ such that $ x \neq y$ and $f(x) = f(y)$.
\item  No finite coloring of the $\cal M$-graph is conflict-free.\label{ert2}
\end{enumerate}
\end{lemma}

\begin{proof}
We will work in $\rca$.  To prove that (\ref{ert1}) implies (\ref{ert2}), suppose that $f: \nat \to k$ is a finite coloring
of the $\cal M$-graph.   Apply $\sf{ERT}$ to find $b$ such that for all $x\ge b$ there is a $y \ge b$ such that $y \neq x$
and $f(x) = f(y)$.  Then every color appearing in $E_b$ appears at least twice.  Thus $f$ is not conflict-free.

To prove the converse, let $f : \nat \to k$ be any function.  We can view $f$ as a $k$-coloring of the $\cal M$-graph.
By (\ref{ert2}), $f$ is not conflict-free.  Thus there is an edge $E_b$ such that every color appearing in $b$ appears at
least twice.  Thus $b$ witnesses that $\sf {ERT}$ holds.
\end{proof}

The principle $\sf{ERT}$ follows trivially from the principle $\sf{ECT}$ which asserts that if $f: \nat \to k$ then
there is a $b$ such that for all $x \ge b$ there are infinitely many values of $y$ such that $f(x) = f(y)$.  By
Theorem 6 of Hirst \cite{ind}, $\sf{ECT}$ is equivalent to $\Sigma_2^0$ induction.  Thus, ${\sf{I}} \Sigma^0_2$ suffices
to prove that no finite coloring of the $\cal M$-graph is conflict-free.

We will show that this use of induction is not necessary by proving $\sf{ERT}$ from $\srt$, the stable Ramsey
theorem for pairs and two colors.
By Corollary 2.6 of Chong, Slaman, and Yang \cite{csy}, $\srt$ cannot prove ${\sf{I}}\Sigma^0_2$, so neither can
$\sf{ERT}$.

\begin{thm}\label{srtert}
$(\rca)$ $\srt$ implies $\sf{ERT}$.
\end{thm}

\begin{proof}
Let $f : \nat \to k$.  Define $g:[\nat ]^2 \to 2 $ by $g(a,b) = 1$ if and only if for some $x \in [a,b)$,
$f(x)$ appears exactly once in the range of $f$ restricted to $[a,b)$.  Because the range of $f$ is $k$,
for fixed $a$ and increasing $x$, the value of $g(a,x)$ can only change at most $2k$ times.  Thus $g$ is a
stable coloring.

Apply $\srt$ to $g$ to obtain an infinite set $H= \{ x_0 , x_1 , x_2 \dots \}$ that is monochromatic for $g$.  Suppose,
by way of contradiction, that $g([H ]^2 ) \equiv 1$.  Let $H_0$ consist of the first $3 \cdot 2^{k-1} $ elements
of $H$.  The elements of $H_0$ define $3\cdot 2^{k-1} -1$ consecutive half-open intervals. Because
$g(x_0, x_{3\cdot 2^{k-1} } ) = 1$, some value of $f$ appears exactly once in $[x_0 , x_{3\cdot 2^{k-1}})$.  Call this
$k_0$.  Either on the left or the right of this location, there must be a collection of $3\cdot 2^{k-2}$ consecutive
elements of $H_0$.  Call these $H_1$.  Note that the range of $f$ on the intervals defined by $H_1$ must omit $k_0$.
Thus, the range of $f$ on these intervals contains at most $k-1$ elements.  Iterating this construction,
$H_{k-1}$ will consist of three consecutive elements of $H_0$ such that $f$ is constant on the union of the two associated
subintervals.  No value of $f$ appears exactly once in the range of $f$ on this union, contradicting the assumption
that $g$ applied to the endpoints yields $1$.
Thus $g([H]^2 ) \equiv 0$ and every value of $f$ appearing at or after $x_0$ must appear at least twice.
\end{proof}

It would be interesting to know how $\sf{ERT}$ compares in strength to other statements that are weaker
than ${\sf I}\Sigma^0_2$, for example, like those described by
Kreuzer and Yokoyama
\cite{ky}

\end{document}